\documentclass{amsart}
\usepackage{graphicx}
\usepackage{amscd}
\usepackage{amsmath}
\usepackage{amsfonts}
\usepackage{amssymb}
\usepackage{fancyhdr}
\theoremstyle{plain}
\newtheorem{theorem}{Theorem}[]
\newtheorem{corollary}[theorem]{Corollary}
\newtheorem{lemma}[theorem]{Lemma}
\newtheorem{proposition}[theorem]{Proposition}
\theoremstyle{definition}
\newtheorem{example}[theorem]{Example}

\newtheorem{remark}[theorem]{Remark}
\theoremstyle{remark}

\numberwithin{equation}{section}

\email{shafiq@cuiatk.edu.pk, shafiq\_ur\_rahman2@yahoo.com}
\email{rubabgull555@gmail.com}
\email{sehr.amin786@yahoo.com}

\begin{document}
\title[\resizebox{4.5in}{!}{A finiteness condition on quasi-local overrings of a class of pinched domains}]{A finiteness condition on quasi-local overrings of a class of pinched domains}

\author{Shafiq ur Rehman, Sehrish Bibi, and Rubab Gull}

\address{}

\address{(Rehman, Bibi, Gull) COMSATS University Islamabad, Attock campus,
Pakistan.}


\begin{abstract}
  An integral domain is called {\em Globalized multiplicatively pinched-Dedekind domain $($GMPD domain$)$} if every nonzero non-invertible ideal can be written as  $JP_1\cdots P_k$ with $J$ invertible ideal and $P_1,...,P_k$ distinct ideals which are  maximal among the nonzero non-invertible ideals, cf. \cite{DumII}.
The GMPD domains with only finitely many overrings have been recently studied in \cite{SU}. In this paper we find the exact number of quasi-local overrings of GMPD domains that only finitely many overrings. Also we study the effect of quasi-local overrings on the properties of GMPD domains. Moreover, we consider the structure of the partially ordered set of prime
ideals (ordered under inclusion) in a GMPD domain.
\end{abstract}

\maketitle

%

\section[Introduction]{Introduction}

A. Jaballah gave necessary and sufficient conditions for the set of overrings of Pr\"{u}fer domains to be finite \cite[Corollary 2.1]{AJ2}. He asked for the exact number of overrings of Pr\"{u}fer domains that have only finitely many overrings and also for the  characterization of domains with finitely many overrings \cite[Question 2.2]{AJ2}. R. Gilmer labeled such domains as FO-domains in \cite{G1}.
Many related results about FO-domains can be found in \cite{G1}, \cite{AJ1}, and \cite{AJ}.\\

 A class of domains, called {Globalized multiplicatively pinched-Dedekind domains $($GMPD domains$)$}, was introduced in \cite{DumII} as an  extension of the class of Dedekind domains and  have been recently studied in \cite {SU} with finiteness condition on the set of its overrings. In this paper we continue to investigate the overring-theoretic properties of GMPD domains and determine the exact number of quasi-local overrings of GMPD domains that only finitely many overrings.
 Further we investigate whether the number of quasi-local overrings affects the properties of GMPD domains. More precisely, if a GMPD domain is given with finitely many quasi-local overrings, then what properties could be characterized by the number of quasi-local overrings? \\

 A short introduction for the notions involved is given here for the reader's convenience. An intermediate ring in the ring extension $A \subseteq B$ is a subring of $B$ that contains $A$.
The set of all intermediate rings in   $A \subseteq B$ is denoted by $[A,B]$. Let $D$ be an integral domain with quotient field $K$. Then
$[D,K]$ denotes the set of all overrings of $D$.  For simplicity we use  $\mathcal{O}(D)$
instead of $[D,K]$. For the set of quasi-local overrings of $D$, we we use the notation $\mathcal{O}_{ql}(D)$. If every prime ideal $P$ in $D$ is strongly prime i.e., whenever $xy \in P$ for some $x,y \in K$ then either $x \in P$ or $y \in P$, then  $D$ is called pseudo-valuation domain (PVD).   Two incomparable valuation domains with the same quotient field are said to be {\em independent} if they have no nonzero common prime ideal. A DVR is a valuation domain with value group $\mathbb{Z}$. A graph that can be drawn in the shape of the letter $Y$ is called
a {\em $Y$-graph} and a graph that  does not contain a $Y$-graph as a subgraph is called {\em $Y$-free}, cf. \cite[Section 2]{AJ1}. If $Spec(D)$ endowed with natural partial
 ordering forms  a tree then  $D$ is called a {\em treed domain},  cf. \cite{G2}. If $Spec(D)$ has no $Y$-graph as a subgraph then $D$ is said to have $Y$-free spectrum.  If each nonzero ideal of $D$ is contained in only finitely many maximal ideals and each nonzero prime ideal is contained in a unique
maximal ideal then $D$
is called {\em h-local}.  A treed domain is h-local if and only if it has $Y$-free spectrum. If $D$ has a unique MNI ideal $Q$ $($by an MNI ideal of
$D$ we mean an ideal of $D$ which is maximal among the non-invertible ideals of $D$$)$ such that every nonzero non-invertible
ideal of $D$ can be factored as $JQ$ for some invertible ideal $J$ then  $D$ is called an MPD domain,  cf. \cite{Dum}. A Dedekind domain is an MPD domain with zero MNI ideal. The quasi-local MPD  domains are: the DVRs, the two-generated pseudo-valuation domains,  the (rank-one) valuation domains with
value group $\mathbb{R}$ and the rank-two strongly discrete valuation domains \cite[Propositions 2.3 and  2.5]{Dum}. If $D$ is $h$-local and all localizations
 of $D$ in maximal ideals are MPD domains then $D$ is called a GMPD domain. Each nonzero non-invertible ideal of $D$ can be
  written as $JQ_1 \cdot Q_2 \cdot Q_3 \cdots Q_n$ for some invertible ideal $J$ and distinct MNI ideals $Q_1,Q_2,Q_3,...,Q_n$ if and only if  $D$ is a GMPD domain, cf.  \cite[Theorem 4]{DumII} and \cite[Theorem 2]{SU1}.
The following implications hold.
$$\begin{array}{ccccccc}
\mbox{Dedekind domain} & \Rightarrow & \mbox{MPD domain}  & \Rightarrow  & \mbox{GMPD domain}\\

\end{array}$$

After summarizing some basic properties of GMPD domains (Proposition \ref{S7}) in section $2$, we prove the following results. For a GMPD domain $D$, $\mathcal{O}(D)$ is finite if and only if  each chain of overrings of $D$ is finite
if and only if  $Spec(D)$ is finite if and only if $Max(D)$ is finite (Proposition \ref{S11}). For a quasi-local GMPD domain $D$,   $\mathcal{O}(D)=\mathcal{O}_{ql}(D)$ (Lemma \ref{1}) . For a GMPD domain $D$,   $\mathcal{O}_{ql}(D)=\cup_{M _i \in Max(D)}\mathcal{O}(D_{M_{i}})$ (Lemma \ref{S2}). If $D$ is a GMPD domain with $|\mathcal{O}(D)|< \infty$,  then
$|\mathcal{O}_{ql}(D)|= \sum_{M_{i} \in Max(D)}|\mathcal{O}(D_{M_{i}})|-|Max(D)|+1$ (Lemma \ref{S4}). If $\mathbb{V}=\{V_{1}, V_{2},...,V_{n}\}$
is the collection of pairwise independent valuation domains  all with quotient field $K$ and if $D=V_{1}\cap V_{2}\cap \cdots \cap V_{n}$ is a
GMPD domain such that for each non-negative integer $n_{i}$,  $1\leq i \leq 3$,  $|\{V\in \mathbb{V} |$ $V$ is $\mathcal{V}_{i}$ $\}|=n_{i}$,
where $\mathcal{V}_{i}$ are valuation domains with value group $G_{i}$,  $G_{1}=\mathbb{R}$, $G_{2}=\mathbb{Z}\times \mathbb{Z}$, and
$G_{3}=\mathbb{Z}$, then $|\mathcal{O}_{ql}(D)|=n_{1}+2n_{2}+n_{3}+1$ (Theorem \ref{S9}). Also $|\mathcal{O}_{ql}(D)|=|Max(D)|+1$ if and
only if no $V_{i}$ has value group $\mathbb{Z}\times\mathbb{Z}$ and $|\mathcal{O}_{ql}(D)|=2|Max(D)|+1$ if and only if all $V_{i}$'s have value group
$\mathbb{Z}\times\mathbb{Z}$ (Corollary \ref{S10}). For any positive integer $n< \infty$, we can construct a GMPD domain having exactly $n$  quasi-local
overrings (Example \ref{9}). If $D$ is GMPD domain with  finite maximal spectrum such that for every non-negative integer
$n_{i}$, $1\leq i\leq 4$, $|\{M \in Max (D)|D_{M}$ is $\mathcal{K}_{i}\}|=n_{i}$, where $\mathcal{K}_{1}=$ two generated PVD but not DVR, $\mathcal{K}_{j}=$
valuation domain with value group $G_{j}, j=2, 3, 4, G_{2}= \mathbb{R}, G_{3}=\mathbb{Z}\times \mathbb{Z}$ and $G_{4}=\mathbb{Z}$,
then $|\mathcal{O}_{ql}(D)|=2(n_{1}+n_{3})+(n_{2}+n_{4})+1$; $D$ is Noetherian if and only if $|\mathcal{O}_{ql}(D)|=2n_{1}+n_{4}+1$;
$D$ is Pr\"{u}fer if and only if $|\mathcal{O}_{ql}(D)|=2n_{3}+(n_{2}+n_{4})+1$; $D$ is Dedekind if and only if $|\mathcal{O}_{ql}(D)|=n_{4}+1$
(Theorem \ref{S5}). If $D$ is an MPD domain with $|Max(D)|=n<\infty$  then $|\mathcal{O}(D)|=2^{n}$ or $3\cdot 2^{n-1}$ and $|\mathcal{O}_{ql}(D)|=n+2$ or $n+1$
(Theorem \ref{ss10}). A GMPD domain $D$ with $|Spec(D)|=n < \infty$   has exactly $\lceil \frac{n}{2} \rceil$ partially ordered sets of  prime ideals
(Theorem \ref{s13}). For a GMPD domain $D$ with $|\mathcal{O}(D)|< \infty$, there exist a Pr\"{u}fer domain $R$ such that $Spec(D) \cong Spec(R)$
as a partially ordered set (Remark \ref{z}).\\

Throughout this paper all rings are (commutative unitary) integral domains. Any unexplained material is standard as in \cite{G} and \cite{Kap}.\\

\section[Main Result]{Main Results}
The first part of this paper deals with overring-theoretic properties of GMPD domains. Some basic facts related to GMPD domains are recalled from \cite{DumII}.

\begin{proposition}\label{S7}
$($\cite[Theorems 6, 9]{DumII}$)$
Let $D$ be a GMPD domain. Then

\begin{enumerate}
\item[(a)] $dim(D) \leq 2$.
\item[(b)] Every maximal ideal contains a unique height-one prime ideal.
\item[(c)] $D$ is a treed domain with  $Y$-free spectrum.
\item[(d)] Every overring of $D$ is a GMPD domain.
\item[(e)] The integral closure $D'$ is a Pr\"{u}fer GMPD domain.
\end{enumerate}
\end{proposition}

\noindent A result analogous to \cite[Corollary 2.1]{AJ2} is attained in the following Proposition under GMPD condition  which is also an improvement of \cite[Proposition 4]{SU}.
\begin{proposition}\label{S11}
For a GMPD domain $D$, the following assertions are equivalent.
\begin{enumerate}
\item[(a)]$\mathcal{O}(D)$ is finite.
\item[(b)] Each chain of overrings of $D$ is finite.
\item[(c)] $Spec(D)$ is finite.
\item[(d)] $Max(D)$ is finite.
\end{enumerate}
\end{proposition}
\begin{proof}
(a)  $\Rightarrow$ (b) and (c) $\Rightarrow$ (d)  are clear. (b) $\Rightarrow$  (c) follows from \cite[Corollary 1.6]{G1}. (d) $\Rightarrow$ (a): As $|\mathcal{O}(D_M)| \leq 3$ for each $M \in Max(D)$, so by \cite[Theorem 3.2]{G1}, $\mathcal{O}(D)$ is finite.
\end{proof}

%
%
%
%

\noindent Next we focus our attention to study the properties of  GMPD domains based on quasi-local overrings.\\

 We denote by  $\mathcal{O}(D)$  the set of all overrings of a domain $D$ and  by $\mathcal{O}_{ql}(D)$ the set of those overrings of  a domain $D$ which are quasi local. Clearly, $\mathcal{O}_{ql}(D) \subseteq \mathcal{O}(D)$ but equality does not hold in general, even if $D$ is quasi-local. For example,  if $K$ is a field and $X,Y$ are indeterminate over $K$ then  the domain $K+YK(X)[[Y]]$ is quasi-local but its overring $K[X]+YK(X)[[Y]]$ is not quasi-local.  Our first result shows that each overring of a  quasi-local GMPD domain is quasi-local.

\begin{lemma}\label{1}
Let $D$ be a quasi-local GMPD domain. Then  $\mathcal{O}(D)=\mathcal{O}_{ql}(D)$.
\end{lemma}
\begin{proof} Each overring of $D$ is either a valuation domain or a 2-generated PVD, cf.  \cite[Definition 2]{DumII} and \cite[Corollary 3.3]{HH}.
\end{proof}

\begin{remark}
 {\em $|\mathcal{O}(D)| < \infty$ if and only if $|\mathcal{O}_{ql}(D)|< \infty$ for each integral domain $D$. Indeed, if $E \in \mathcal{O}(D)$ then  $E=\bigcap \{E_M \mid M \in Max(E)\}$ where $E_M \in \mathcal{O}_{ql}(D)$ for  each   $M \in Max(E)$. Hence, if $|\mathcal{O}_{ql}(D)|< \infty$ then $|\mathcal{O}(D)| < \infty$.}
\end{remark}

\begin{lemma}\label{S2}
For a GMPD domain $D$, $\mathcal{O}_{ql}(D)=\bigcup \{\mathcal{O}(D_{M_{i}}) \mid M _i \in Max(D)\}$.
\end{lemma}
\begin{proof}
  Using Lemma \ref{1}, we can easily obtain that $\mathcal{O}(D_{M_i})=\mathcal{O}_{ql}(D_{M_i}) \subseteq \mathcal{O}_{ql}(D)$. Further, suppose that  $(E, M) \in \mathcal{O}_{ql}(D)$. If $Q=M \cap D$, then  $Q \subseteq M_{i}$ for some $M_i \in Max(D)$ and so $D_{M_{i}}\subseteq D_{Q} \subseteq E$. This implies $E \in \mathcal{O}(D_{M_{i}})$.
\end{proof}

\noindent Note that a  domain $D$ with finite maximal spectrum is h-local if and only if $D_MD_N$ equals the quotient field of $D$, for every
two distinct maximal ideals $M$ and $N$ of $D$.
\begin{lemma}\label{S4}
Let $D$ be a GMPD domain with $|\mathcal{O}(D)|< \infty$. Then
$$|\mathcal{O}_{ql}(D)|=\sum \limits_{M_{i} \in
 Max(D)}|\mathcal{O}(D_{M_{i}})|-|Max(D)|+1$$
\end{lemma}
\begin{proof}
Since $D$ is h-local, so $|\mathcal{O}(D_{M})\cap \mathcal{O}(D_{N})|=1$ for every
two distinct maximal ideals $M$ and $N$ of $D$. Now apply Lemma \ref{S2}.
\end{proof}

Recall \cite[Section 22]{G} that two incomparable valuation domains with the same quotient field are said to be {\em independent} if they have no non-zero common prime ideal. Equivalently, two valuation domains with the same quotient field are said to be {\em independent} if there exist no non-trivial valuation overring containing the both. More precisely, if $V_{1}$ and $V_{2}$ are valuation domains with the same quotient field $K$, then $V_{1}$ and $V_{2}$ are {\em independent} if and only if $V_{1}V_{2}=K$ if and only if no non-zero prime ideal of $V_{1}\cap V_{2}$ survives in both $V_{1}$ and $V_{2}$. Let $V_{1}, V_{2},...,V_{n}$ be pairwise independent valuation domains all with quotient field $K$. Then $D=V_{1}\cap V_{2}\cap \cdots \cap V_{n}$ is GMPD  if and only if each $V_{i}$ has value group $\mathbb{Z}$, $\mathbb{Z}\times\mathbb{Z}$ or $\mathbb{R}$, cf. \cite[Theorem 6]{SU}. In the next result we count the quasi-local overrings of those GMPD domains which are obtained by intersection of valuation domains.\\

Let $\mathbb{V}=\{V_{1}, V_{2},...,V_{n}\}$ be the collection of pairwise independent valuation domains  all with quotient field $K$ and let $D=V_{1}\cap V_{2}\cap \cdots \cap V_{n}$ be a GMPD domain such that for each non-negative integer $n_{i}$,  $1\leq i \leq 3$,  $|\{V\in \mathbb{V} |$ $V$ is $\mathcal{V}_{i}$ $\}|=n_{i}$, where $\mathcal{V}_{i}$ are valuation domains with value group $G_{i}$,  $G_{1}=\mathbb{R}$, $G_{2}=\mathbb{Z}\times \mathbb{Z}$, and $G_{3}=\mathbb{Z}$.
\begin{theorem}\label{S9}
With notation above $|\mathcal{O}_{ql}(D)|=n_{1}+2n_{2}+n_{3}+1$.
\end{theorem}
\begin{proof}
As $V_i's$ are pairwise independent, so $D$ is  h-local and B\'{e}zout, cf. \cite[Section 3]{O}.
Let $P_i$ be the center of $V_i$ on $D$. Then $Max(D)=\{P_1,P_2,...,P_n\}$ and $D_{P_i}=V_i$, cf. \cite[Theorem 107]{Kap} or \cite[Corollary 2]{B}. Now apply Lemma \ref{S4} and the fact that  a valuation domain of finite dimension $d$ has $d+1$ overrings.
\end{proof}
\begin{corollary}$\,$ \label{S10}
With notation above;
\begin{enumerate}
\item[(a)] $|\mathcal{O}_{ql}(D)|=n+1$ if and only if no $V_{i}$ has value group $\mathbb{Z}\times\mathbb{Z}$.
\item[(b)] $|\mathcal{O}_{ql}(D)|=2n+1$ if and only if all $V_{i}$'s have value group $\mathbb{Z}\times\mathbb{Z}$.
\end{enumerate}
\end{corollary}

  Now for any positive integer $n< \infty$,
we are able to construct a GMPD domain having exactly $n$  quasi-local overrings, as illustrated in the next example. Recall \cite[Section 43]{G} that $D$ is a  Krull domain
if $D=\cap_{P \in X^1(D)}D_P$, this intersection has finite character and $D_P$ is a DVR  for each $P \in X^{1}(D)$,
where $X^1(D)$ is the set of height-one prime ideals of $D$. Clearly a UFD is a Krull domain, cf. \cite[Proposition 43.2]{G}.

\begin{example}\label{9}
{\em Let $D$ be a Krull domain with quotient field $K$,  $\{P_i\}_{i=1}^{n-1}$ be a finite collection of height-one prime ideals of $D$ and
let  $S=D-\cup_{i=1}^{n-1} P_i$. Then $D_S$ is a GMPD domain having exactly $n$ quasi-local overrings.
Indeed, because $\{D_{P_i}\}_{i=1}^{n-1}$  are DVRs with the same quotient field $K$ and $D_S= \cap_{i=1}^{n-1} D_{P_i}$.}
\end{example}

Recall \cite[Corollary 19]{DumII} that for every cardinal numbers $c_i$, $1\leq i\leq 4$,  there exists a GMPD domain $D$ such that the set $\{M \in Max(D)\ |\ D_M$ is $\mathcal{K}_i\}$ has cardinality $c_i$, where $\mathcal{K}_{1}=$ two generated PVD but not DVR, $\mathcal{K}_{j}=$ valuation domain with value group $G_{j}, j=2, 3, 4, G_{2}= \mathbb{R}, G_{3}=\mathbb{Z}\times \mathbb{Z}$ and $G_{4}=\mathbb{Z}$. Recall \cite[Lemma 2]{SU} that a non-integrally closed two-generated PVD $D$ with quotient field $K$ has exactly three overrings $D,$ $D^{'},$ and $K$.

\begin{theorem}\label{S5}
Let $D$ be a GMPD domain with $|Max(D)|< \infty$ such that for every non-negative integer $n_i$, $1\leq i\leq 4$, $\big|\{M \in Max(D) \mid D_M$ is $\mathcal{K}_i\}\big|=n_i$, where $\mathcal{K}_1=$ two-generated PVD but not DVR, $\mathcal{K}_j=$ valuation domain with value group $G_j$, $j=2,3,4$, $G_2=\mathbb{R}$, $G_3=\mathbb{Z}\times \mathbb{Z}$ and $G_4=\mathbb{Z}$. Then
\begin{enumerate}
	\item [(a)] $|\mathcal{O}_{ql}(D)|=2(n_{1}+n_{3})+(n_{2}+n_{4})+1$.
	
	\item [(b)] $D$ is Noetherian if and only if $|\mathcal{O}_{ql}(D)|=2n_{1}+n_{4}+1$.
	
	\item [(c)] $D$ is Pr\"{u}fer if and only if $|\mathcal{O}_{ql}(D)|=2n_{3}+(n_{2}+n_{4})+1$.

    \item [(d)] $D$ is Dedekind if and only if $|\mathcal{O}_{ql}(D)|=n_{4}+1$.

\end{enumerate}
\end{theorem}
\begin{proof}
Apply Lemma \ref{S4} and the facts that a non-integrally closed two-generated PVD   has exactly three overrings and a valuation domain of
finite dimension $d$ has $d+1$ overrings.
\end{proof}

Next we find the exact number of overrings and quasi-local overrings of an MPD domain. Recall that an MPD domain is a GMPD domain with unique MNI ideal,
 cf. \cite{Dum}.

\begin{theorem}\label{ss10}
Let $D$ be an MPD domain with $|Max(D)|=n < \infty$. Then
\begin{enumerate}
	\item [(a)] $|\mathcal{O}(D)|=2^{n}$ or $3\cdot 2^{n-1}$.
\item [(b)] $|\mathcal{O}_{ql}(D)|=n+2$ or $n+1$.
\end{enumerate}
\end{theorem}
\begin{proof}
 (a): Let $Max(D)=\{M_1, M_2, M_3, ..., M_n\}$. From \cite[Propositin 2.7]{Dum}, there exist a maximal ideal  $M_i$ such that  $D_{M_i}$ is MPD and $D_{M_j}$ is  DVR for each $j \neq i$.
If $D_{M_{i}}$ is not a DVR, then  $D_{M_{i}}$ is either a $2$-generated PVD or a valuation
domain with value group $\mathbb{R}$ or $\mathbb{Z}\times \mathbb{Z}$. Therefore, $|\mathcal{O}(D_{M_{i}})|=2$ or $3$ and $|\mathcal{O}(D_{M_{j}})|=2$ for each $i \neq j$. Hence by  \cite[Theorem 10]{SU1}, we get that  $|\mathcal{O}(D)|=3\cdot2^{n-1}$.  If $D_{M_{i}}$ is a DVR, then again by  \cite[Theorem 10]{SU1}, we get that $|\mathcal{O}(D)|=2^{n}$.

(b):  Apply \cite[Propositin 2.7]{Dum} and Lemma \ref{S4}.
\end{proof}

 At the end we consider the structure of the partially ordered set of prime
ideals (ordered under inclusion) in a GMPD domain. For a GMPD domain $D$ with $|Spec(D)| < \infty$, we consider the question that what
partially ordered sets could be arise as $Spec(D)$? Keeping in view the basic properties of GMPD domain, given in Proposition \ref{S7},  we make first the
 following observations:   \\

\begin{center}
\begin{tabular}{ | c | p{7cm} | }
\hline $|Spec(D)|$ &  $Spec(D)$ for a GMPD domain D   \\ \hline

1 & $\bullet$ \\ \hline
      2 &
\begin{picture}(40,33)(-90,-20)\put(-70,-15){\line(0,1){20}}\put(-72,2){$\bullet$}\put(-72,-17){$\bullet$}\end{picture}
 \\ \hline
     3& \begin{picture}(40,33)(-90,-20)\put(-90,-16){\line(0,1){22}}\put(-92,-8){$\bullet$}\put(-92,-20){$\bullet$}\put(-92,4){$\bullet$}\put(-49,-20){\line(1,3){8}} \put(-49,-20){\line(-1,3){8}}      \put(-51,-21){$\bullet$}\put(-43,3){$\bullet$}\put(-60,3){$\bullet$}\end{picture}

\\ \hline 4&
\begin{picture}(40,33)(-90,-20)
\put(-90,-16){\line(0,1){22}}
\put(-88,-18){\line(1,2){5}}
\put(-85,-9){$\bullet$}
\put(-92,-8){$\bullet$}
\put(-92,-20){$\bullet$}
\put(-92,4){$\bullet$}
\put(-19,-20){\line(1,1){18}} 
\put(-19,-20){\line(-1,1){18}}      
\put(-21,-21){$\bullet$}
\put(-2,-3){$\bullet$}
\put(-40,-3){$\bullet$}
\put(-18,-18){\line(0,1){18}}
\put(-21,-3){$\bullet$}
\end{picture}
 \\ \hline
    5 &
\begin{picture}(40,33)(-90,-20)
\put(-20,-16){\line(0,1){22}}
\put(-19,-18){\line(1,2){5}}
\put(-21,-18){\line(-1,2){5}}
\put(-29,-9){$\bullet$}
\put(-15,-9){$\bullet$}
\put(-22,-8){$\bullet$}
\put(-22,-20){$\bullet$}
\put(-22,4){$\bullet$}
\put(42,-20){\line(1,2){10}}
\put(41,-20){\line(-1,2){10}}
\put(39,-21){$\bullet$}
\put(50,-3){$\bullet$}
\put(39,-3){$\bullet$}
\put(42,-20){\line(1,1){18}}
\put(58,-3){$\bullet$}
\put(42,-18){\line(0,1){17}}
\put(30,-3){$\bullet$}
 \put(-87,-5){\line(0,1){12}}
 \put(-73,-5){\line(0,1){12}}
 \put(-79,-18){\line(1,2){5}}
 \put(-81,-18){\line(-1,2){5}}
 \put(-89,-9){$\bullet$}
 \put(-76,-9){$\bullet$}
 \put(-89,4){$\bullet$}
\put(-75,4){$\bullet$}
 \put(-82,-20){$\bullet$}
\end{picture}
\\ \hline
   6 &
 \begin{picture}(40,33)(-90,-20)
       \put(-20,-19){\line(5,3){18}}
       \put(-4,-10){$\bullet$}

 \put(42,-20){\line(1,2){10}}
 \put(41,-20){\line(-1,2){10}}
 \put(39,-21){$\bullet$}
 \put(50,-3){$\bullet$}
 \put(39,-3){$\bullet$}
 \put(42,-20){\line(1,1){18}}
 \put(58,-3){$\bullet$}
 \put(42,-18){\line(0,1){17}}
 \put(30,-3){$\bullet$}
  \put(-87,-5){\line(0,1){12}}
  \put(-73,-5){\line(0,1){12}}
\put(-80,-17){\line(0,1){12}}
  \put(-79,-18){\line(1,2){5}}
  \put(-81,-18){\line(-1,2){5}}
\put(42,-20){\line(5,3){28}}
\put(68,-4){$\bullet$}
\put(-82,-9){$\bullet$}
  \put(-89,-9){$\bullet$}
  \put(-76,-9){$\bullet$}
  \put(-89,4){$\bullet$}
\put(-75,4){$\bullet$}
 \put(-82,-20){$\bullet$}
 \put(-20,-16){\line(0,1){22}}
 \put(-19,-18){\line(1,2){5}}
 \put(-21,-18){\line(-1,2){5}}
 \put(-29,-9){$\bullet$}
 \put(-15,-9){$\bullet$}
 \put(-22,-8){$\bullet$}
 \put(-22,-20){$\bullet$}
\put(-22,4){$\bullet$}
 \end{picture}
\\ \hline
\end{tabular}
\end{center}
$\\$
$\\$
After these observations a natural question arises that how many structurally distinct partially ordered sets of prime ideals
a GMPD domain can have? The answer to this question is provided in the following proposition.
\begin{proposition}\label{s13}
Let $D$ be a GMPD domain with $|Spec(D)|=n < \infty$. Then $D$ can have atmost $\lceil \frac{n}{2} \rceil$ partially ordered sets of  prime ideals.
\end{proposition}
\begin{proof}
 We can assume that $n > 1$. Since  $D$ is a treed domain with
$Y$-free spectrum, so $ n-1 \leq 2 |Max(D)|$ and hence $\frac{n-1}{2} \leq |Max(D)| \leq n-1$. If $n$ is odd, the possibilities for
the number of maximal ideals in each spectrum is $n-1, n-2, ....,\frac{n-1}{2}$ respectively. If $n$ is even, the possibilities for
the number of maximal ideals in each spectrum is $n-1, n-2, ....,\frac{n}{2}$ respectively. Also each prime spectrum has distinct cardinality of maximal ideals.
Hence the total possibilities for the number of distinct partially ordered sets of prime ideals is $\lceil \frac{n}{2} \rceil$.
\end{proof}

 \begin{remark}\label{z} Any two partially ordered sets $U$ and $V$ are said to be isomorphic if there is an order preserving bijection $f: U \rightarrow V$ such that $f^{-1}$ is also order
preserving. By \cite[Theorem 3.1]{L} and Proposition \ref{S7}, we can easily deduce that for a GMPD domain $D$  with $|\mathcal{O}(D)|< \infty$ there exist a Pr\"{u}fer  domain $R$ such that $Spec(D) \cong Spec(R)$ (as a partially ordered set).
\end{remark}

\noindent {\bf Acknowledgements.}
We are thankful to  Higher Education Commission of Pakistan for supporting and facilitating this research.


\begin{thebibliography}{11}










\bibitem{Dum} {T. Dumitrescu and S. U. Rahman,}
{A class of pinched domains,}  Bull. Math. Soc. Sci. Math. Roumanie  {\bf 52} (2009), 41-55.


\bibitem{DumII} {T. Dumitrescu and S. U. Rahman,}
{A class of pinched domains II,} Comm. Alg. {\bf 39} (2011), 1394-1403.


\bibitem{G1}{R. Gilmer,}
{ Some finiteness conditions on the number of overrings of an integral domain,}
{Proc. Amer. Math. Soc.  {\bf 131}(8) (2002), 2337-2346.}


\bibitem{G} {R. Gilmer,}
{\em Multiplicative Ideal Theory}, Marcel Dekker, New York, 1972.




\bibitem{HH} { J. R. Hedstrom and E. G. Houston,}
{Pseudo-valuation domains,}
Pacific J. Math. {\bf 75} (1978), 137-147.

\bibitem{AJ1}{A. Jaballah,}
{Numerical characterizations of some integral domains,} Monatshefte für Mathematik, {\bf 164} (2011), 171-181.



\bibitem{AJ} {A. Jaballah,}
{The number of overrings of an integrally closed domains,}  Expo. Math. {\bf 23} (2005), 353-360.

\bibitem{AJ2}{A. Jaballah,}{ Finiteness of the set of intermediary rings in normal pairs,}
 Saitama Math. J., {\bf 17} (1999), 59–61.

 \bibitem{Kap} {I. Kaplansky,}
{\em Commutative Rings}, rev. ed. The University of Chicago Press,
Chicago and London, 1974.

\bibitem{L} {W. J.  Lewis,}
{ The spectrum of a ring as a partially ordered set,}
 J. Alg. , {\bf 25} (1973), 419-434.



\bibitem{O}{B. Olberding,}
{Characterizations and constructions of h-local domains,}
In Contributions to Module Theory (Walter de Gruyter, Berlin, 2008), 385-406.


\bibitem{G2}{G. Picavet,}
{Treed domains,}
Int. Elect. J. Algebra, {\bf 3} (2008), 43-57.


\bibitem{B} {B. Prekowitz,}
{Intersection of quasi-local domains,}  Trans.  Amer. Math. Soc.  {\bf 181} (1973), 329-339.



 \bibitem{SU1} {S. U. Rehman,}
       {A note on a characterization theorem for a certain class of domains,} Mis. Math. Notes,  {\bf 18} (2017), 427-429.

       \bibitem{SU} {S. U. Rehman,}
       {A finiteness condition on the set of overrings of some classes of integral domains,} J. Alg. Appl., {\bf 16} (2018), (9 pages).
		

$\\$

\end{thebibliography}
\end{document}